\documentclass[11pt]{amsart}
\author{Christian J. Berghoff}
\title{Efficient computation of universal elliptic Gau{\ss} sums}
\address{Universit\"at Bonn, Mathematisches Institut, Endenicher Allee 60, 53115 Bonn, Germany}
\email{berghoff@math.uni-bonn.de}

\usepackage{lmodern}
\usepackage[T1]{fontenc}
\usepackage[utf8x]{inputenc}
\usepackage{a4wide}
\usepackage[english]{babel}
\usepackage{amsmath, amssymb, amsthm, amsfonts}
\usepackage{framed}
\usepackage{enumerate}
\usepackage{url}

\usepackage{algorithmic}

\algsetup{linenodelimiter=.}

\makeatletter
\newcounter{algorithm}
\renewcommand{\thealgorithm}{\arabic{algorithm}}
\def\algorithm{\@ifnextchar[{\@algorithma}{\@algorithmb}}
\def\@algorithma[#1]{%
	\refstepcounter{algorithm}
	\trivlist
	\leftmargin\z@
	\itemindent\z@
	\labelsep\z@
	\item[\parbox{\columnwidth}{%
		\hrule
		\hrule
		\noindent\strut\textbf{Algorithm \thealgorithm.} #1
		\hrule
	}]\hfil\vskip0em%
}
\def\@algorithmb{\@algorithma[]}

\makeatother

\usepackage{mathabx}
\usepackage[all]{xy}
\usepackage{textcomp}
\usepackage{stmaryrd}
\usepackage{nicefrac}
\usepackage[numbers]{natbib}
\usepackage[pdftex]{graphicx}
\usepackage{layout}
\usepackage{fancyhdr}
\usepackage{setspace}
\usepackage{color}
\usepackage{hyperref}

\usepackage{float}
\usepackage{pgfplots}
\pgfplotsset{compat=1.8}
\usepackage{subfigure}

\numberwithin{equation}{section}

\theoremstyle{plain}

\newtheorem{satz}{Theorem}[section]
\newtheorem{prop}[satz]{Proposition}
\newtheorem{lemma}[satz]{Lemma}
\newtheorem{koro}[satz]{Corollary}

\theoremstyle{definition}

\newtheorem{defi}[satz]{Definition}

\theoremstyle{remark}

\DeclareMathOperator{\gal}{Gal}

\DeclareMathOperator{\deg2}{deg}
\DeclareMathOperator{\ord}{ord}

\DeclareMathOperator{\M}{\mathsf{{M}}}
\DeclareMathOperator{\Cpoly}{\mathsf{C}}

\DeclareMathOperator{\lc}{lc}
\let\prec\relax 
\DeclareMathOperator{\prec}{prec}

\newcommand{\OO}{\mathcal{O}}
\newcommand{\Q}{\mathbb{Q}}
\newcommand{\F}{\mathbb{F}}

\newcommand{\N}{\mathbb{N}}
\newcommand{\Z}{\mathbb{Z}}

\newcommand{\C}{\mathbb{C}}

\newcommand{\An}{\mathbf{A}}

\begin{document}
\begin{abstract}
In \cite{ELGS} it has been shown that the elliptic Gau{\ss} sums whose use has been proposed in the context of counting points on elliptic curves and primality tests in \cite{Mi_CIDE, CIDE} can be computed by using modular functions. In this work we give detailed algorithms for the necessary computations mentioned in \cite{ELGS}, all of which have been implemented in C. We analyse the relatively straightforward algorithms derived from the theory and provide several improvements speeding up computations considerably. In addition, slightly generalizing \cite{ELGS} we describe how (elliptic) Jacobi sums may be determined in a very similar way and show how this can be used. We conclude by an analysis of space and run-time requirements of the algorithms.
\end{abstract}
\maketitle

\tableofcontents
\section{Elliptic curves}\label{sec:ell_kurven}
Within this work we will only consider primes $p>3$ and thus assume that the curve in question is given in the Weierstra{\ss} form
\[ 
E: Y^2=X^3+aX+b=f(X),
 \]
where $a, b \in \F_p$. We will always identify $E$ with its set of points $E(\overline{\F_p})$. For the following well-known statements cf. \cite{Silverman, Washington}. We assume that the elliptic curve is neither singular nor supersingular. It is a standard fact that $E$ is an abelian group with respect to point addition. Its neutral element, the point at infinity, will be denoted $\OO$. For a prime $\ell \neq p$, the $\ell$-torsion subgroup $E[\ell]$ has the shape
\[ 
E[\ell]\cong \Z/\ell\Z \times \Z/\ell\Z.
 \]
In the endomorphism ring of $E$ the Frobenius homomorphism
\[ 
\phi_p: (X, Y) \mapsto (\varphi_p(X), \varphi_p(Y))=(X^p, Y^p)
 \]
satisfies the quadratic equation 
\begin{equation}\label{eq:char_gl}
0=\chi(\phi_p)=\phi_p^2-t\phi_p+p,
 \end{equation} 
where $|t|\leq 2\sqrt{p}$ by the Hasse bound. By restriction $\phi_p$ acts as a linear map on $E[\ell]$. The number of points on $E$ over $\F_p$ is given by $\#E(\F_p)=p+1-t$ and is thus immediate from the value of $t$.\\
Schoof's algorithm computes the value of $t$ modulo $\ell$ for sufficiently many small primes $\ell$ by considering $\chi(\phi_p)$ modulo $\ell$ and afterwards combines the results by means of the Chinese Remainder Theorem. In the original version this requires computations in extensions of degree $O(\ell^2)$.\\
However, a lot of work has been put into elaborating improvements. Let $\Delta=t^2-4p$ denote the discriminant of equation \eqref{eq:char_gl}. Then we distinguish the following cases:
\begin{enumerate}
\item If $\left( \frac{\Delta}{\ell}\right)=1$, then $\ell$ is called an \textit{Elkies prime}. In this case, the characteristic equation factors as $\chi(\phi_p)=(\phi_p-\lambda)(\phi_p-\mu) \mod \ell$, so when acting on $E[\ell]$ the map $\phi_p$ has two eigenvalues $\lambda, \mu \in \F_\ell^*$ with corresponding eigenpoints $P, Q$. Since $\lambda\mu=p$ and $\lambda+\mu=t$, it obviously suffices to determine one of them. So we have to solve the discrete logarithm problem
\[ 
\lambda P=\phi_p(P)=(P_x^p, P_y^p),
 \]
which only requires working in extensions of degree $O(\ell)$.
\item If  $\left( \frac{\Delta}{\ell}\right)=-1$, then $\ell$ is called an \textit{Atkin prime}. In this case the eigenvalues of $\phi_p$ are in $\F_{\ell^2}\backslash\F_\ell$ and there is no eigenpoint $P \in E[\ell]$. We do not consider this case.
\end{enumerate}

The approach to Elkies primes was further improved in numerous publications, e.~g. \cite{MaMu, GaMo, BoMoSaSc, Enge, BrLaSu, Sutherland}. We focus on the new ideas introduced in \cite{MiMoSc}. The algorithm it presents allows to work in extensions of degree $n$, where $n$ runs through maximal coprime divisors of $\ell-1$, using so-called \textit{elliptic Gaussian periods}.\\

A variant of this approach was presented in \cite{Mi1} and \cite{MiVu}. It relies instead on so-called \textit{elliptic Gau{\ss} sums}. For a character $\chi: (\Z/\ell\Z)^* \rightarrow \langle\zeta_n\rangle$ of order $n$ with $n \mid \ell-1$ these are defined in analogy to the classical cyclotomic Gau{\ss} sums via
\begin{equation}\label{eq:ell_gs}
G_{\ell,n, \chi}(E)=\sum_{a=1}^{\ell-1}\chi(a)(aP)_u
\end{equation}
for an $\ell$-torsion point $P$ on $E$, where $u=y$ for $n$ even and $u=x$ for $n$ odd. As was shown in \cite{Mi1},
\begin{equation}\label{eq:ell_gs_pot_eigenschaft}
G_{\ell,n, \chi}(E)^n, \frac{G_{\ell,n, \chi}(E)^m}{G_{\ell,n, \chi^m}(E)} \in \F_p[\zeta_n]\quad \text{for}\quad m<n
\end{equation}
holds. In addition, the index in $\F_\ell^*$ of the eigenvalue $\lambda$ corresponding to $P$ can directly be calculated modulo $n$ using the equation
\begin{equation}\label{eq:lambda_aus_ell_gs}
G_{\ell,n, \chi}(E)^p=\chi^{-p}(\lambda)G_{\ell, n, \chi^p}(E)\quad \Rightarrow \quad 
\frac{G_{\ell, n, \chi}(E)^m}{G_{\ell, n, \chi^m}(E)}(G_{\ell, n, \chi}(E)^n)^q=\chi^{-m}(\lambda),
\end{equation}
where $p=nq+m$ holds. When the quantities from equation \eqref{eq:ell_gs_pot_eigenschaft} have been computed, it thus suffices to do calculations in the extension $\F_p[\zeta_n]$ of degree $\varphi(n)$ to derive the index of $\lambda$ in $\F_\ell^*$ modulo $n$ before composing the modular information by means of the Chinese remainder theorem. In the following sections we will be concerned with the efficient computation of the quantities in question using \textit{universal elliptic Gau{\ss} sums}, which were defined in \cite[Corollary 2.25]{ELGS}, instead of using the definition \eqref{eq:ell_gs}, which requires passing through larger extensions.
\section{Computation of the universal elliptic Gau{\ss} sums}\label{sec:rechnungen}
\subsection{Prerequisites}\label{sec:prerequisites}
We first recall some facts from \cite{ELGS}, to which we refer the reader for details. A modular function of weight $k \in \Z$ for a subgroup $\Gamma^\prime \subseteq \textup{SL}_2(\Z)$ is a meromorphic function $f(\tau)$ on the upper complex half-plane $\mathbb{H}=\{\tau \in \C: \Im(\tau)>0\}$ satisfying
\begin{equation}\label{eq:def_mod_funk}
f(\gamma\tau)=(c\tau+d)^kf(\tau)\ \textup{for all}\ \gamma=\left(\begin{smallmatrix}
a & b \\ c & d
\end{smallmatrix}\right) \in \Gamma^\prime,
\end{equation}
where $\gamma\tau=\frac{a\tau+b}{c\tau+d}$, and some technical conditions. Equation \eqref{eq:def_mod_funk} in particular implies $f$ can be written as a Laurent series in terms of $q_N=\exp\left(\frac{2\pi i\tau}{N}\right)$ for some $N \in \N$ depending on $\Gamma^\prime$. We use the notation $q=q_1$ and consider the groups 
$\Gamma^\prime=\Gamma_0(\ell):=\left\{\left(\begin{smallmatrix}
a & b\\ c& d
\end{smallmatrix}\right) \in \textup{SL}_2(\Z): \ell \mid c \right\}$. The field of modular functions of weight $0$ for a group $\Gamma^\prime$ will be denoted by $\An_0(\Gamma^\prime)$. The Fricke-Atkin-Lehner involution $w_\ell$ acts on modular functions $f(\tau)$ via $f(\tau) \mapsto f\left(\frac{-1}{\ell\tau}\right)=:f^*(\tau)$, where $f^*(\tau)=f(\ell\tau)$ for $f(\tau) \in \An_0(\textup{SL}_2(\Z))$ holds. We recall the Laurent series
\begin{align}
x(w,q)&=\frac{1}{12}+\frac{w}{(1-w)^2}+\sum_{n=1}^{\infty}\sum_{m=1}^{\infty}mq^{nm}(w^m+w^{-m})
-2mq^{nm}, \label{eq:x_def}\\
y(w, q)&=\frac{w+w^2}{2(1-w)^3}+\frac{1}{2}\sum_{n=1}^{\infty}\sum_{m=1}^{\infty}\frac{m(m+1)}{2}
\left(q^{nm}(w^m-w^{-m})+q^{n(m+1)}(w^{m+1}-w^{-(m+1)}) \right), \label{eq:y_def}\\
\eta(q)&=q^{\frac{1}{24}}\left(1+\sum_{n=1}^{\infty}(-1)^n\left(q^{n(3n-1)/2}+q^{n(3n+1)/2}\right)\right),\label{eq:eta_def}\\
m_\ell(q)&=\ell^s \left(\frac{\eta(q^\ell)}{\eta(q)}\right)^{2s}\quad \text{with}\quad s=\frac{12}{\gcd(12, \ell-1)}.\label{eq:ml_def}
\end{align}
We further use $p_1(q)=\sum_{\zeta \in \mu_\ell, \zeta \neq 1} x(\zeta, q)$, the modular discriminant $\Delta(q)=\eta(q)^{24}$ and the well-known $j$-invariant $j(q)$, which is surjective on $\C$. There is a polynomial $M_\ell \in \C[X, Y]$, sometimes referred to as the canonical modular polynomial, such that $M_\ell(X, j(q))$ is irreducible over $\C(j(q))[X]$ and $m_\ell(q)$ is one of its roots. Furthermore, $\deg2 M_\ell=\frac{\ell-1}{\gcd(\ell-1,\, 12)}$ holds.\\
Now for a prime $\ell$ let $n \mid \ell-1$ and $\chi: \F_\ell^* \rightarrow \mu_n$ be a character of order $n$. Defining 
\begin{equation}\label{eq:G_ln_lr}
G_{\ell, n, \chi}(q)=G_{\ell, n}(q):=\sum_{\lambda \in \F_\ell^*} \chi(\lambda)V(\zeta_\ell^{\lambda}, q)\quad 
\text{with}\quad V=\begin{cases}
x,\quad n \equiv 1 \mod 2,\\
y, \quad n \equiv 0 \mod 2,
\end{cases}
 \end{equation}
corollary 2.25 and proposition 2.16 of \cite{ELGS} imply the following

\begin{satz}\label{satz:gs_modulfunktion_2}
Let $\ell$, $n$, $\chi$ be as above. Furthermore, let
\[ 
r=
\begin{cases}
\min\left\{r: \frac{n+r}{6} \in \N\right\},& n \equiv 1 \mod 2,\\ 
3,& n=2, \\
0,& \text{else}
\end{cases}
\text{and}\quad 
e_{\Delta}=\begin{cases}
\frac{n+r}{6}, & n\equiv 1 \mod 2,\\
1, & n=2,\\
\frac{n}{4}, & \text{else.}
\end{cases}
 \]
Define the \textup{universal elliptic Gau{\ss} sum}
\begin{equation}\label{eq:tau_ln_lr}
\tau_{\ell, n}(q):=\frac{G_{\ell, n}(q)^n p_1(q)^r}{\Delta(q)^{e_{\Delta}}}.
 \end{equation}
Then $\tau_{\ell, n}(q)$ has coefficients in $\Q[\zeta_n]$ and is a holomorphic modular function of weight $0$ for $\Gamma_0(\ell)$. There exist $k\geq 0$ and a polynomial $Q(X, Y) \in \C[X, Y]$ with $\deg2_Y(Q)<\deg2_Y(M_\ell)$ such that
\begin{equation}\label{eq:gs_darstellung_ml}
\tau_{\ell, n}(q)\frac{\partial M_\ell}{\partial Y}(m_\ell(q), j(q))=m_\ell(q)^{-k}Q(m_\ell(q), j(q)).
 \end{equation}
\end{satz}
If we can efficiently compute the polynomial $Q$ in \eqref{eq:gs_darstellung_ml}, we obtain a rational expression for $\tau_{\ell, n}(q)$ in terms of $j(q)$ and $m_\ell(q)$. As detailed in \cite[Section 3]{ELGS}, for an elliptic curve $E/\F_p$ this formula readily translates to a formula for the elliptic Gau{\ss} sum $G_{\ell, n}(E)^n$ from \eqref{eq:ell_gs} in terms of the values $j(E)$, $m_\ell(E)$ which may be efficiently computed.

\subsection{Laurent series}\label{sec:implementierung}

\subsubsection{General remarks}
When dealing with Laurent series $g(q)=\sum_{i=i_0}^\infty g_iq^i$ in this section we will write $\ord(g)=i_0$ for the order and $\lc(g)=g_{i_0}$ for the leading coefficient of the series. We first remark that when computing the Laurent series of the universal elliptic Gau{\ss} sums $\tau_{\ell, n}(q)$ the part in formula \eqref{eq:x_def} which is independent from $w$ vanishes due to the well-known properties of character sums.\\

In addition, in our implementation using the GMP library \cite{GMP} we wished to use as long as possible the data type for integers for efficiency reasons, and only to convert our results into rational numbers in the last step. In order to realise this plan, we analyse the denominators of the Laurent series of the universal elliptic Gau{\ss} sums.\\
We first remark that in formula \eqref{eq:x_def} used to compute the $x$-coordinate all coefficients except for the constant one, which is independent from $q$, lie in $\Z[\zeta_\ell]$ and that they lie in $\frac{1}{2}\Z[\zeta_\ell]$ for formula \eqref{eq:y_def} corresponding to the $y$-coordinate. The constant term yields a power possibly dividing the denominator of the expression. The exact value of $v_\ell$ is computed in
\begin{lemma}\label{lem:l_exp}
Let $v_\ell$ be the exponent of $\ell$ in the denominator of the Gau{\ss} sum $\tau_{\ell, n}(q)$. Then we obtain
\[ 
v_\ell\leq
\begin{cases}
\lceil \frac{2n}{\ell-1} \rceil, &n \text{ odd,}\\
\lceil \frac{3n}{\ell-1} \rceil, &n \text{ even.}
\end{cases}
 \]

\end{lemma}
\begin{proof}
As is well-known,
\[ 
(\ell)=(1-\zeta_\ell)^{\ell-1}
 \]
holds when both sides are considered as ideals in $\Z[\zeta_\ell]=\OO(\Q[\zeta_\ell])$. This implies
\[ 
\ell(1-\zeta_\ell)^{-k}=e(1-\zeta_\ell)^{\ell-1}(1-\zeta_\ell)^{-k}=e(1-\zeta_\ell)^{\ell-1-k}
 \]
for $e \in \Z[\zeta_\ell]^*$. Since $(1-\zeta_\ell) \notin \Z[\zeta_\ell]^*$, this yields
\[ 
\ell(1-\zeta_\ell)^{-k} \in \Z[\zeta_\ell] \Leftrightarrow k \leq \ell-1.
 \]
More generally, one obtains
\[ 
\ell^{v_\ell}(1-\zeta_\ell)^{-k}=e(1-\zeta_\ell)^{v_\ell(\ell-1)-k} \in \Z[\zeta_\ell] \Leftrightarrow k \leq v_\ell(\ell-1) \Leftrightarrow v_\ell \geq \frac{k}{\ell-1}.
 \]
Hence, $v_\ell$ can be taken to be $\lceil \frac{k}{\ell-1} \rceil$. The shape of the constant terms in formulae \eqref{eq:x_def} and \eqref{eq:y_def} as well as theorem \ref{satz:gs_modulfunktion_2} yield the assertion.
\end{proof}

We remark that the coefficients of $12p_1(q)$ are likewise integers. Since $\Delta(q)$ has leading coefficient $1$, the coefficients of $\Delta(q)^{-1}$ are also integers. From these considerations and the definition of $\tau_{\ell,n}(q)$ in theorem \ref{satz:gs_modulfunktion_2} we deduce
\begin{koro}\label{koro:gs_nenner}
Let $\ell$ be a prime and let $n \mid \ell-1$. Let $r=\min\{r: \frac{n+r}{6} \in \N\}$ for $n \equiv 1 \mod 2$ be defined as in theorem \ref{satz:gs_modulfunktion_2} and $v_{\ell}$ as in lemma \ref{lem:l_exp}. Define
\[ 
c=\begin{cases}
12^r\ell^{v_{\ell}},\quad & n \equiv 1 \mod 2,\\
2^2\cdot12^3\ell^{v_{\ell}}, \quad & n=2,\\
2^n\ell^{v_{\ell}},\quad & \text{else}.
\end{cases}
 \]
Then the coefficients of $c\cdot\tau_{\ell,n}(q)$ lie in $\Z[\zeta_n]$.
\end{koro}

\subsubsection{An improved algorithm}
As is evident from \eqref{eq:G_ln_lr}, $G_{\ell,n}(q) \in \Q[\zeta_n, \zeta_\ell]((q))$ holds. In order to compute $\tau_{\ell, n}(q)$ up to precision $\prec(\ell,n)$ we need to determine, in particular, the $n$-th power of an element of this field, which requires
\[ 
O(\log n \M(\ell n \prec(\ell, n)))
 \]
multiplications in $\Z$, if we use the multiple of $\tau_{\ell, n}(q)$ from corollary \ref{koro:gs_nenner}. Equation \eqref{eq:prec} implies we can choose $\prec(\ell, n)=\ell(e_{\Delta}+v+1)$. In the worst case $n, v \in O(\ell)$ holds, which yields a rapidly growing run-time of $O(\log n \M(\ell^4))$ for this step. We now wish to show how this run-time can be significantly reduced.\\

In particular, we make use of the following 
\begin{lemma}\label{lem:produkt_in_kleinerem_koerper}
Let $G_{\chi}(\zeta_\ell)=\sum_{\lambda \in \F_\ell^*} \chi(\lambda)\zeta_\ell^{\lambda}$ denote the ordinary cyclotomic Gau{\ss} sum. Then we obtain
\[ 
G_{\ell, n, \chi}(q)G_{\chi^{-1}}(\zeta_\ell) \in \Q[\zeta_n]((q)).
 \]
\end{lemma}
\begin{proof}
By definition the expression in question lies in $\Q[\zeta_\ell, \zeta_n]((q))$. We consider the action of the galois group $G$ of the field extension $\Q[\zeta_\ell, \zeta_n]/\Q[\zeta_n]$ on this expression. Since $(\ell, n)=1$ holds, we obviously obtain $G=\langle \sigma: \zeta_\ell \mapsto \zeta_\ell^c \rangle$, where $c$ is a generator of $\F_\ell^*$. As shown in the proof of corollary 2.25 in \cite{ELGS},
\[ 
\sigma(G_{\ell, n, \chi}(q))=\chi^{-1}(c)G_{\ell, n, \chi}(q)
 \]
holds. In the same vein one can show 
\[ 
\sigma(G_{\chi}(\zeta_\ell))=\chi^{-1}(c)G_{\chi}(\zeta_\ell),
 \]
which immediately implies the invariance of the expression in question under $\sigma$.
\end{proof}

Using this lemma we modify the algorithm for determining $\tau_{\ell, n}(q)$, which results from formulae \eqref{eq:x_def}, \eqref{eq:y_def}, \eqref{eq:G_ln_lr} and \eqref{eq:tau_ln_lr}, in the following way. Instead of directly computing the $n$-th power of $G_{\ell,n,\chi}(q)$, we calculate:

\begin{algorithm}[Fast computation of $G_{\ell, n, \chi}(q)^n$]
\label{algo:gs_lr_besser}
\begin{algorithmic}
    \REQUIRE{$\ell, n, \prec(\ell, n)$}
    \ENSURE{$G_{\ell, n, \chi}(q)^n$}

	\STATE Compute $G_{\ell, n, \chi}(q)$ up to precision $\prec(\ell, n)$ using formulae \eqref{eq:x_def}, \eqref{eq:y_def}, \eqref{eq:G_ln_lr}.\\
    \STATE $T_1:=G_{\ell, n, \chi}(q)G_{\chi^{-1}}(\zeta_\ell)$\\
	\STATE $T_2:=T_1^n$\\
	\STATE $T_3:=G_{\chi^{-1}}(\zeta_\ell)^n$\\
	\RETURN $T_4=T_2T_3^{-1}$
\end{algorithmic}
\end{algorithm}

It is obvious that $T_4=G_{\ell, n, \chi}(q)^n$ holds and we thus compute the result we wish for. According to lemma \ref{lem:produkt_in_kleinerem_koerper} the run-time for step 2 is merely
\[ 
O(\log n \M(n \prec(\ell,n)))=O(\log n \M(\ell^3)),
 \]
since a whole power of $\ell$ is saved in the degree of the polynomials to be multiplied. It is likewise easy to see that $T_3$ can be computed in run-time $O(\log n \M(\ell n))$ and $T_4$ in run-time $O(\prec(\ell, n)\M(n))$, which is negligible.\\

The determination of $T_1$, however, requires $O(\prec(\ell,n)\M(\ell n))$ operations when using a naive implementation. We show how to reduce this cost considerably by computing

\begin{align*}
G_{\ell, n, \chi}(q)G_{\chi^{-1}}(\zeta_\ell)=\sum_{\substack{\lambda_1, \lambda_2 \in \F_\ell^*}} \chi(\lambda_1)V(\zeta_\ell^{\lambda_1}, q)\chi^{-1}(\lambda_2)\zeta_\ell^{\lambda_2}
=\sum_{c=\lambda_1\lambda_2^{-1} \in \F_\ell^*} \chi(c) \sum_{\lambda_1 \in \F_\ell^*} \zeta_\ell^{\lambda_1c^{-1}}V(\zeta_\ell^{\lambda_1}, q).
\end{align*}

Before further transforming the inner sum, we remark
\[ 
V(\zeta_\ell, q)=\sum_{i=0}^{\infty}q^i\sum_{k=0}^{\ell-1} a_{i, k}\zeta_\ell^k \Rightarrow V(\zeta_\ell^{\lambda_1}, q)=\sum_{i=0}^{\infty}q^i\sum_{k=0}^{\ell-1} a_{i, k}\zeta_\ell^{\lambda_1 k},
 \]
which follows by applying an appropriate power of the automorphism  $\sigma: \zeta_\ell \mapsto \zeta_\ell^c$. Hence, one obtains
\begin{align*}
&\sum_{\lambda_1 \in \F_\ell^*}\zeta_\ell^{\lambda_1c^{-1}}V(\zeta_\ell^{\lambda_1}, q)
=\sum_{i=0}^{\infty}q^i\sum_{\lambda_1 \in \F_\ell^*}\zeta_\ell^{\lambda_1c^{-1}} \sum_{k=0}^{\ell-1} a_{i, k}\zeta_\ell^{\lambda_1 k}
=\sum_{i=0}^{\infty}q^i\sum_{\lambda_1 \in \F_\ell^*}\sum_{k=0}^{\ell-1} a_{i, k}\zeta_\ell^{\lambda_1 (k+c^{-1})}\\
=&\sum_{i=0}^{\infty}q^i\left(\sum_{\lambda_1 \in \F_\ell^*}a_{i, -c^{-1}} + 
\sum_{\substack{k=0\\k \neq -c^{-1}}}^{\ell-1} a_{i, k} \sum_{\lambda_3=\lambda_1(k+c^{-1}) \in \F_\ell^*} \zeta_\ell^{\lambda_3}\right)
=\sum_{i=0}^{\infty}q^i\left((\ell-1)a_{i, -c^{-1}} + 
\sum_{\substack{k=0\\k \neq -c^{-1}}}^{\ell-1} a_{i, k} \sum_{\lambda_3=1}^{\ell-1}\zeta_\ell^{\lambda_3} \right)\\
=&\sum_{i=0}^{\infty}q^i\left((\ell-1)a_{i, -c^{-1}} -\sum_{\substack{k=0\\k \neq -c^{-1}}}^{\ell-1} a_{i, k}\right)
=:\sum_{i=0}^{\infty}b_i(c)q^i,
\end{align*}
where the last equality holds because of $\sum_{i=0}^{\ell-1}\zeta_\ell^i=0$. In total, this yields
\[ 
G_{\ell, n, \chi}(q)G_{\chi^{-1}}(\zeta_\ell)=\sum_{i=0}^{\infty}q^i\sum_{c\, \in\, \F_\ell^*}\chi(c)b_i(c)
 \]
with $b_i(c) \in \Q$. Furthermore, for $c_1, c_2 \in \F_\ell^*$
\begin{align}\label{eq:b_i_update}
b_i(c_1)=b_i(c_2)+\ell(a_{i, -c_1^{-1}}-a_{i, -c_2^{-1}})
 \end{align}
obviously holds, whence the computation of $b_i(c), c \in \F_\ell^*$, requires a run-time of $O(\ell)$ for fixed $i$. In order to compute $T_1$ up to the required precision, we thus proceed as follows:

\begin{algorithm}[Speed-up of step 1 in algorithm \ref{algo:gs_lr_besser}]
\label{algo:gs_lr_schritt1_schnell}
\begin{algorithmic}[1]
    \REQUIRE{$\ell, n, \prec(\ell, n)$}
    \ENSURE{$G_{\ell, n, \chi}(q)G_{\chi^{-1}}(\zeta_\ell)$}

	\STATE Determine the coefficients $a_{i,k}$ of $V(\zeta_\ell, q)$ using formulae \eqref{eq:x_def} and \eqref{eq:y_def}, respectively, for $i=0, \ldots, \prec(\ell, n), k=0, \ldots, \ell-1$.\\
	\STATE For $i=0, \ldots, \prec(\ell,n)$ determine the values $b_i(c), c \in \F_\ell^*$, using \eqref{eq:b_i_update}.\\
	\RETURN $\sum_{i=0}^{\prec(\ell, n)}q^i\sum_{c \in \F_\ell^*}\chi(c)b_i(c)$.
\end{algorithmic}
\end{algorithm}

Using this algorithm the value $T_1$ can be determined using $O(\ell \prec(\ell, n))$ operations, which, for large $n$, is negligible as compared to the amount required by the second step of the new algorithm.\\

To avoid the less efficient GMP data type for rational numbers when computing $T_4$ in algorithm \ref{algo:gs_lr_besser} we again examine by which factor we have to multiply intermediate results to ensure all coefficients are integers. The required statement is furnished by
\begin{lemma}
Let $\ell$ be a prime and $\chi$ be a character of order $n \mid \ell-1$. Then one obtains
\[ 
\ell^n G_{\chi}(\zeta_\ell)^{-n} \in \Z[\zeta_n].
 \]
\end{lemma}
\begin{proof}
The identity
\[ 
\sigma(G_{\chi}(\zeta_\ell))=\chi^{-1}(c)G_{\chi}(\zeta_\ell)
 \]
shown in lemma \ref{lem:produkt_in_kleinerem_koerper}, where $\sigma: \zeta_\ell \mapsto \zeta_\ell^c$ and $\langle c \rangle = \F_\ell^*$, implies
\[ 
G_{\chi}(\zeta_\ell)^n \in \Q[\zeta_n].
 \]
Using the definition of the Gau{\ss} sums it is evident that $G_{\chi}(\zeta_\ell)^n$ actually lies in $\Z[\zeta_n]$. Furthermore, a general property exhibited by Gau{\ss} sums (cf. \cite[p.~91]{Shimura}) is
\[ 
G_{\chi}(\zeta_\ell)G_{\chi^{-1}}(\zeta_\ell)=\chi(-1)\ell.
 \]
Raising this equation to the $n$-th power and using $\chi(-1)^n=1$, we obtain
\[ 
G_{\chi}(\zeta_\ell)^n G_{\chi^{-1}}(\zeta_\ell)^n=\ell^n
 \]
and hence
\[ 
\ell^nG_{\chi}(\zeta_\ell)^{-n}=G_{\chi^{-1}}(\zeta_\ell)^n \in \Z[\zeta_n].
 \]
\end{proof}


\begin{koro}\label{koro:laufzeit_lr}
The universal elliptic Gau{\ss} sum $\tau_{\ell, n}(q)$ can be computed using
\[ 
O(\log n\M(n\prec(\ell, n)))
 \]
multiplications in $\Z$.
\end{koro}


\subsection{Rational expression}\label{sec:erg}
In this section we present algorithms by means of which it will be possible to compute the polynomial $Q$ from theorem \ref{satz:gs_modulfunktion_2} and thus to determine a rational expression in terms of $j(q)$ and $m_\ell(q)$ for the universal elliptic Gau{\ss} sums $\tau_{\ell, n}$ once their Laurent series have been computed as discussed in section \ref{sec:implementierung}.\\

Slightly rewriting equation \eqref{eq:gs_darstellung_ml} we obtain the representation
\begin{align}\label{eq:zaehler_poly_formel}
\tau_{\ell,n}(q)\frac{\partial M_\ell}{\partial Y}(m_\ell(q), j(q))=Q(m_\ell(q), j(q)),
 \end{align}
where $Q(X, Y)=\sum_{i=i_{min}}^{i_{max}}\sum_{k=0}^{v-1} c_{i, k}X^iY^k$ with $\deg2_Y(M_\ell)=v=\frac{\ell-1}{\gcd(\ell-1,\, 12)}$ according to \cite[pp.~61--62]{Mueller} holds. The polynomial $M_\ell$ can be determined using algorithm 5.8 from \cite{Mueller}. The left hand side of this equation can be computed up to a certain precision using the known $q$-expansions and the algorithms from section \ref{sec:implementierung}. As is evident from the definition of $m_\ell$ as well as $\eta$, the Laurent series $m_\ell(q)$ has order $v$. Furthermore, $\ord(j)=-1$ holds. This implies $\ord(m_\ell^ij^k)=iv-k$. Hence, all summands of the expression
\[ 
Q(m_\ell(q), j(q))=\sum_{i=i_{min}}^{i_{max}}\sum_{k=0}^{v-1}c_{i, k}m_\ell^ij^k
 \]
exhibit different orders, which allows to compute the coefficients $c_{i, k}$ successively using the precomputed left hand side of equation \eqref{eq:zaehler_poly_formel}. Full details are given in \cite{Diss}, where it is also shown that it suffices to compute all occurring Laurent series up to precision $(v+e_{\Delta})\ell$.\\

It turns out, however, that for complexity reasons it makes sense to replace the function $m_\ell(q)$ by another modular function $a_\ell(q)$ with similar properties. This function is mentioned in \cite{Mueller, Enge} as an alternative to $m_{\ell}(q)$ in the context of Schoof's algorithm. Its associated (minimal) polynomial $A_\ell(X, j(q))$ indeed exhibits a significantly smaller degree in $j$ and significantly smaller coefficients than $M_\ell(X, j(q))$ and was used, for example, in achieving the point-counting record described in \cite{Rekord}.\\

We first require the following
\begin{defi}
Let $f(\tau)$ be a modular function of weight $1$ and $r$ a prime. The $r$-th Hecke operator $T_r$ acts on $f$ via
\[ 
T_r(f)(\tau)=\frac{1}{r}\sum_{k=0}^{r-1}f\left(\frac{\tau+k}{r}\right)+f(r\tau).
 \]
\end{defi}

\begin{lemma}\cite[p.~74]{Mueller}\label{lem:funk_a_eig}
Let $\ell>3$, $s=\frac{24}{\gcd(24, \ell+1)}$ and let $r$ be an odd prime satisfying
\[ 
s \mid r-1, \quad \left( \frac{r}{\ell} \right)=1, \quad \left(\frac{\ell}{r}\right)=1.
 \]
Then the function
\begin{equation}\label{eq:A_def}
a_\ell(\tau)=\frac{T_r(\eta(\tau)\eta(\ell\tau))}{\eta(\tau)\eta(\ell\tau)}
 \end{equation}
is a modular function of weight $0$ for $\Gamma_0(\ell)$ which is holomorphic on $\mathbb{H}$. Furthermore, $a_\ell(\tau)$ is invariant under the Fricke-Atkin-Lehner involution $w_\ell$.
\end{lemma}

In order to perform actual computations using $a_\ell(\tau)$ we need its Laurent series up to the required precision. For its determination we use formula  \eqref{eq:eta_def} for the Laurent series of the $\eta$-function and subsequently apply
\begin{prop}\cite[p.~74]{Mueller}\label{prop:hecke_lr}
Let 
\[ 
a(\tau)=\exp\left(2 \pi i \tau \frac{z}{s}\right)\sum_{k=0}^{\infty}a_k \exp(2\pi ik \tau)
 \]
be the Laurent series of a function $a(\tau)$, where $\gcd(z, s)=1$.
Then
\begin{align}\label{eq:Tr_lr}
T_r(a(\tau))=\exp\left(2 \pi i \tau \frac{z}{s} \right)& \left( \sum_{\substack{k=0\\r\, \mid\, ks+z}}^{\infty} a_k \exp\left(2 \pi i \tau \frac{ks+(1-r)z}{rs}\right) \right.\nonumber \\
& \left. {} + \sum_{k=0}^{\infty} a_k \exp\left(2 \pi i \tau kr+\frac{(r-1)z}{s} \right) \right)
\end{align}
holds.\\
\end{prop}

We now write $v:=-\ord(a_\ell)$. The minimal polynomial $A_\ell(X, j(\tau)) \in \C[j(\tau)][X]$ of $a_\ell(\tau)$ has the form
\begin{equation}\label{eq:mipo_a_l}
A_\ell(X, j(\tau))=\sum_{i=0}^{\ell+1}\sum_{k=0}^{2v} a_{i, k}X^ij(\tau)^k
 \end{equation}
according to \cite[p.~77]{Mueller}. An easy calculation shows that the congruence $ks+z \equiv 0 \mod r$ implies $v=0$ for $\ell<29$ as well as for $\ell \in \{37, 43, 67, 163\}$. 
Thus, in this case the minimal polynomial $A_\ell(X, j)$ has degree $0$ in $j$. Hence, $a_\ell(\tau) \in \C$ holds and for these mostly small values of $\ell$ we still have to resort to $m_\ell(\tau)$ for computations. To determine the minimal polynomial $A_\ell(X, j(\tau))$ of $a_\ell(\tau)$ we use the algorithm presented in \cite[p.~79]{Mueller}.\\

We first prove an analogue to proposition 2.16 of \cite{ELGS} and equation \eqref{eq:gs_darstellung_ml}.
\begin{prop}\label{prop:darst_a_l}
Let $\ell, n, \chi$ be as in section \ref{sec:prerequisites}. Then for a holomorphic modular function $g(\tau)$ of weight $0$ for $\Gamma_0(\ell)$, in particular for $\tau_{\ell, n}$, there exists a polynomial $Q(X, Y) \in \C[X, Y]$ with $\deg2_Y(Q)<\deg2_Y(A_\ell)$ such that
\begin{equation}\label{eq:darst_a_l}
g(\tau)\frac{\partial A_\ell}{\partial Y}(a_\ell(\tau), j(\tau))=Q(a_\ell(\tau),j(\tau))
\end{equation}
holds.
\end{prop}
\begin{proof}
Applying corollary 2.15 of \cite{ELGS} and setting $f(\tau)=a_\ell(\tau)$, it suffices to show that $a_\ell: \mathbb{H} \rightarrow \C$ is surjective, since this implies the holomorphic functions in $\C(a_\ell(\tau))$ are given by $\OO=\C[a_\ell(\tau)]$.\\

By definition, $A_\ell(a_\ell(\tau), j(\tau))=0$ holds. Applying the Fricke-Atkin-Lehner involution $w_\ell$ to this equation we obtain that $a_\ell(\tau)$ is also a root of $A_\ell(X, j(\ell\tau))$. Equation \eqref{eq:mipo_a_l} yields
\begin{equation}\label{eq:a_l_mipo_ell_sym}
\sum_{i=0}^{\ell+1}X^i\sum_{k=0}^{2v} a_{i, k}j(\ell\tau)^k=A_\ell(X, j(\ell\tau))=\sum_{i=0}^{\ell+1}s_{\ell+1-i}(\tau)X^i,
 \end{equation}
where $s_{\ell+1-i}(\tau)$ are the elementary-symmetric polynomials in the roots $a_\ell(\tau)=f_0(\tau), \ldots, f_\ell(\tau)$ of $A_\ell(X, j(\ell\tau))$. In \cite[p.~77]{Mueller} it is shown the Laurent series of these functions have the orders
\[ 
\ord(f_i)=-v,\ 0\leq i<\ell, \quad \ord(f_\ell)=-\ell v,
 \]
from which we conclude
\[
\ord(s_0)=0,\quad ord(s_{\ell+1-i})=-\ell v-(\ell-i)v=-(2\ell-i)v,\ 0\leq i \leq \ell.
 \]
Using $\ord(j(\ell\tau))=-\ell$ equation \eqref{eq:a_l_mipo_ell_sym} implies $a_{0, 2v}\neq 0$ and $a_{i, 2v}=0$ for $i>0$. So $A_\ell(c, Y)$ is a polynomial of degree $2v$ in $Y$ for any $c \in \C$. Due to the surjectivity of $j$ there exists $\tau \in \mathbb{H}$ such that $A_\ell(c, j(\tau))=0$. Now the roots of $A_\ell(X, j(\tau))$ are well known to be $a_\ell(S_k(\tau))$ for a set of representatives $S_k$, $0\leq k \leq \ell$, of $\textup{SL}_2(\Z)/\Gamma_0(\ell)$, cf. Lemma 2.11 of \cite{ELGS}. Thus, for one of these matrices $S_k$ the identity $c=a_\ell(S_k\tau)$ holds. Hence, $a_\ell$ attains all values $c \in \C$.
\end{proof}

Equation \eqref{eq:darst_a_l} implies the enumerator of the rational expression for $\tau_{\ell, n}$ consists of monomials of the form $a_\ell^ij^k$ with $0\leq k \leq 2v-1$. However, $\ord(a_\ell^ij^k)=-iv-k=\ord(a_\ell^{i-1}j^{k+v})$ obviously holds for $k<v$. Hence, we are faced with pairs of two monomials of equal order. This implies the coefficients $c_{i, k}$ of the polynomial $Q$ cannot be successively computed, which is the case when using $m_\ell$ as pointed out above. Of course they can still be obtained by inverting the matrix corresponding to the linear system defined by equation \eqref{eq:darst_a_l}. However, this would require significantly higher costs.
In order to solve this problem we use the following
\begin{lemma}\label{lem:zerlegung_wl}
Let $f(\tau)$ be a function which is anti-invariant under the action of $w_\ell$, i.~e., let $f^{\ast}=-f$. Let $g(\tau) \in \mathbf{A}_0(\Gamma_0(\ell))$ not be invariant under $w_\ell$. Then
\[ 
g(\tau)=\frac{g(\tau)+g^{\ast}(\tau)}{2}+\frac{(g(\tau)-g^{\ast}(\tau))f(\tau)}{2f(\tau)}=:g^{(1)}(\tau)+\frac{g^{(2)}(\tau)}{f(\tau)}
 \]
holds if $f(\tau)\neq 0$, where $g^{(1)}(\tau), g^{(2)}(\tau)$ are invariant under $w_\ell$.
\end{lemma}
\begin{proof}
Evident from the prerequisites, since $g^{**}=g$.
\end{proof}
To determine the rational expression for $\tau_{\ell, n}(q)$, we proceed as follows:

\begin{algorithm}[Determining the rational expression for $\tau_{\ell, n}$ using $a_\ell$]
\label{algo:algo_zerlegung_wl}
\begin{algorithmic}[1] 
    \REQUIRE{$\ell, n, \prec(\ell, n)$}
    \ENSURE{Rational expression for $\tau_{\ell, n}(q)$}

	\STATE Determine $g \in \mathbf{A}_0(\Gamma_0(\ell))$ with $g^{\ast}=-g$.\label{algo:schritt_anti_inv_funk}\\
	\STATE Determine $\tau_{\ell, n}$ and $\tau_{\ell, n}^*$ up to precision $\prec(\ell, n)$ using algorithm \ref{algo:gs_lr_besser} and the formulae from section \ref{sec:prerequisites} and equation \eqref{eq:wl_x}.\\
 	\STATE Compute $\tau_{\ell, n}^{(1)}, \tau_{\ell, n}^{(2)}$ according to lemma \ref{lem:zerlegung_wl}.\\
	\STATE For both functions determine rational expressions $R_1$, $R_2$ in terms of $a_\ell$ and $j$.\label{algo:schritt_ausdruck_a}\\
	\STATE Compute $\tau_{\ell, n}:=R_1(a_\ell, j)+\frac{R_2(a_\ell, j)}{g}$.
\end{algorithmic}
\end{algorithm}

We have to specify how the steps \ref{algo:schritt_anti_inv_funk} and \ref{algo:schritt_ausdruck_a} are done in practice. Concerning the determination of $g$ we remark that
$g(\tau):=m_\ell^{\ast}(\tau)-m_\ell(\tau)$ is anti-invariant under $w_\ell$ and thus fulfils the conditions we require. In addition, it has a relatively small order, which is relevant from complexity reasons. Furthermore, for every $g$ with these properties $w_\ell(g^2)=g^2$ obviously holds. Hence, we can compute a rational expression for $g^2$ in terms of $a_\ell$ and $j$ using algorithm \ref{algo:algo_ausdruck_wl} below. The value for $g$ itself is calculated on the elliptic curve in question, as is detailed in section \ref{sec:elkies_gs}. Before presenting an algorithm for step \ref{algo:schritt_ausdruck_a}, we investigate the orders of several Laurent series.
\begin{lemma}\label{lem:schranke_abl_wl}
Let $A_\ell(X, j)$ be the minimal polynomial of $a_\ell$. Hence,
\begin{equation}\label{eq:mipo_a_l_2}
0=A_\ell(a_\ell, j)=\sum_{i=0}^{\ell+1}\sum_{k=0}^{2v} a_{i, k}a_\ell^i j^k
 \end{equation}
holds according to \cite[p.~77]{Mueller}, where $\ord(a_\ell)=-v$. Then the following statements hold true:
\begin{enumerate}
\item $\ord\left(\frac{\partial}{\partial Y}A_\ell(a_\ell, j) \right)\geq -(\ell+1)v+1.$
\item $\ord \left(\frac{\partial}{\partial Y}A_\ell(a_\ell, j(q^\ell)) \right) \geq -(2v-1)\ell.$
\end{enumerate}
\end{lemma}
\begin{proof}
First of all we examine which of the coefficients of $A_\ell$ do not vanish. Since $\ord(a_\ell^ij^k)=-iv-k$ holds, searching for two summands $a_\ell^{i_1}j^{k_1}$, $a_\ell^{i_2}j^{k_2}$ of equal order leads to the equation
\[ 
(i_2-i_1)v=(k_1-k_2).
 \]
As this implies $v \mid k_1-k_2$, depending on the values of $k_1, k_2$ the equation exhibits the solutions $(i, 0), (i-1, v), (i-2, 2v)$ and $(i, k), (i-1, k+v)$ for $0 < k < v$. Paying heed to the restrictions for possible values of $i$ and $k$ this directly implies $\ord(a_\ell^ij^k)\geq -(\ell+1)v$ for all summands whose coefficients do not vanish. 
Hence,
\begin{equation}\label{eq:Mipo_A_koeff_schranke}
a_{i, k}\neq 0\quad \Rightarrow\quad -iv-k \geq -(\ell+1)v
\end{equation}
holds. We now investigate
\[ 
\frac{\partial}{\partial Y}A_\ell(a_\ell, j)=\sum_{i=0}^{\ell+1}\sum_{k=0}^{2v-1} (k+1)a_{i, k+1}a_\ell^ij^k.
 \]
From \eqref{eq:Mipo_A_koeff_schranke} we deduce
\[ 
a_{i, k+1} \neq 0 \quad \Rightarrow\quad -iv-(k+1)\geq -(\ell+1)v \quad\Rightarrow\quad \ord(a_\ell^ij^k)=-iv-k\geq -(\ell+1)v+1.
 \]
The second assertion is proved analogously after applying $w_\ell$ to equation \eqref{eq:mipo_a_l_2}.
\end{proof}

Let $f(\tau) \in \mathbf{A}_0(\Gamma_0(\ell))$ be invariant under the action of $w_\ell$. According to proposition \ref{prop:darst_a_l} there exists a polynomial $Q \in \C[X, Y]$ with $\deg2_Y(Q)<2v$ such that
\begin{equation}\label{eq:ausdruck_A}
f(\tau)\frac{\partial}{\partial Y}A_\ell(a_\ell, j)=Q(a_\ell, j)\quad \text{and}
\quad f(\tau)\frac{\partial}{\partial Y}A_\ell(a_\ell, j(q^\ell))=Q(a_\ell, j(q^\ell)),
 \end{equation}
where the second equation arises from the application of $w_\ell$ to the first one.
To determine the rational expression using $a_\ell$ and $j$ we thus use the following

\begin{algorithm}[Determining the rational expression for $f(\tau)$ using $a_\ell$]
\label{algo:algo_ausdruck_wl}
\begin{algorithmic}[1]
    \REQUIRE{$\ell, n, f(\tau)$}
    \ENSURE{$Q(X, Y)=\sum_i \sum_k q_{i, k}X^iY^k$ from equation \eqref{eq:ausdruck_A}}

	\STATE Set $Q:=0$, $\prec(\ell, n):=-\ord(f)+(\ell+1)v$.\label{algo:algo_wl_prec}\\
	\STATE Compute $j(q)$, $a_\ell(q)$ up to precision $\prec(\ell, n)$ using formulae \eqref{eq:eta_def}, \eqref{eq:A_def} and \eqref{eq:Tr_lr}.\\
	\STATE Compute $A_\ell(X, Y)$ using algorithm 5.26 from \cite{Mueller}.\\
	\STATE Compute $s_1:=f(\tau)\frac{\partial}{\partial Y}A_\ell(a_\ell, j)$ and $s_2:=f(\tau)\frac{\partial}{\partial Y}A_\ell(a_\ell, j(q^\ell))$ up to precision $\prec(\ell, n)$.\\
	\STATE Set $p_1:=\ord(s_1)$, $p_2:=\ord(s_2)$.\\
	\WHILE{$s_1\neq 0$ \label{algo:algo_wl_schleife}}
		\STATE $o_1:=p_1, o_2:=p_2$\\
		\WHILE{$p_1<o_1+\ell-1$\label{algo:algo_wl_schleife_1}}
			\STATE Determine $(i_1, k_1), (i_2, k_2)$ satisfying $\ord(a_\ell^{i_s}j^{k_s})=p_1$, $s=1, 2$, with $0\leq k_1<v$ and $k_2=k_1+v$.\label{algo:algo_wl_ord_1}\\
			\STATE Compute $s_1:=s_1-q_{i_2, k_2}a_\ell^{i_2}j^{k_2}$.\label{algo:algo_wl_koeff_1}\\
			\STATE Compute $s_1:=s_1-\frac{\lc(s_1)}{\lc(a_\ell^{i_1}j^{k_1})}a_\ell^{i_1}j^{k_1}$. Set $Q:=Q+\frac{\lc(s_1)}{\lc(a_\ell^{i_1}j^{k_1})}X^{i_1}Y^{k_1}$.\\
			\STATE $p_1:=p_1+1$\\	
		\ENDWHILE
		\WHILE{$p_2<o_2+\ell-1$\label{algo:algo_wl_schleife_2}}
			\STATE Determine $(i_1, k_1), (i_2, k_2)$ satisfying $\ord(a_\ell^{i_s}j(q^\ell)^{k_s})=p_2$, $s=1, 2$, with $0\leq k_1<v$ and $k_2=k_1+v$.\label{algo:algo_wl_ord_2}\\
			\STATE Compute $s_2:=s_2-q_{i_1, k_1}a_\ell^{i_1}j(q^\ell)^{k_1}$. \label{algo:algo_wl_koeff_2}\\
			\STATE Compute $s_2:=s_2-\frac{\lc(s_2)}{\lc(a_\ell^{i_2}j(q^\ell)^{k_2})}a_\ell^{i_2}j(q^\ell)^{k_2}$. Set $Q:=Q+\frac{\lc(s_2)}{\lc(a_\ell^{i_2}j(q^\ell)^{k_2})}X^{i_2}Y^{k_2}$.\\
			$p_2:=p_2+1$\\
		\ENDWHILE
	\ENDWHILE
	\RETURN $Q$.
\end{algorithmic}
\end{algorithm}

The idea of the algorithm consists in considering the equations $s_1=Q(a_\ell, j)$ as well as 
$s_2=s_1^\ast=Q(a_\ell, j^\ast)$ in turns. The restrictions satisfied by the orders of the Laurent series allow to compute sets of $\ell-1$ coefficients at a time using successive elimination, since for one of the two summands of equal order the coefficient has either already been computed or vanishes. This is shown in detail in

\begin{lemma}\label{lem:algo_wl_korrekt}
Algorithm \ref{algo:algo_ausdruck_wl} works correctly.
\end{lemma}
\begin{proof}
To prove this we have to show the following statements:
\begin{enumerate}
\item The specified precision suffices to find the polynomial $Q$.
\item The coefficients $q_{i_2, k_2}$ and $q_{i_1, k_1}$ used in steps \ref{algo:algo_wl_koeff_1} and \ref{algo:algo_wl_koeff_2}, respectively, have already been computed unless they vanish.
\end{enumerate}
Concerning the first point we remark again that according to proposition \ref{prop:darst_a_l} $Q(X, Y)$ contains only positive powers of $a_\ell$. Hence, the order $o$ of its summands satisfies
\begin{equation}\label{eq:prec_algo_ausdruck}
\ord\left(f(\tau)\frac{\partial}{\partial Y}A_\ell(a_\ell, j) \right)=\ord(f)-(\ell+1)v+1 \leq o \leq 0,
 \end{equation}
where the equality follows from lemma \ref{lem:schranke_abl_wl}. Since the value $p_1$ ranging over the order of $s_1$ strictly increases in each iteration of loop \ref{algo:algo_wl_schleife_1}, the specified precision
\[ 
\prec(\ell, n)=-\ord(f)+(\ell+1)v 
 \]
is sufficient.\\
Concerning the second issue we observe that according to lemma \ref{lem:schranke_abl_wl} after $t$ iterations of loop \ref{algo:algo_wl_schleife}
\[ 
o_1\geq o_1(t)
:=\ord(f)-(\ell+1)v+1+(\ell-1)t, \quad o_2\geq o_2(t):=\ord(f)-(2v-1)\ell+(\ell-1)t
 \]
holds. Furthermore, one calculates $o_1(t)=o_2(t)+(\ell-1)(v-1)$. Now assume that in step \ref{algo:algo_wl_ord_1} 
$o_1(t)\leq p_1=-i_2v-k_2<o_1(t)+\ell-1$ holds. Using $k_2=k_1+v$, this yields
\begin{align*}
& o_1(t)-(\ell-1)k_2\leq -i_2v-k_2\ell<o_1(t)-(\ell-1)(k_2-1)\\
 \overset{k_2=k_1+v}{\Rightarrow}
& o_2(t)-(\ell-1)(k_1+1)\leq -i_2v-k_2\ell<o_2(t)-(\ell-1)k_1.
\end{align*}
Now $k_1\geq 0$ implies $-i_2v-k_2\ell=p_2$ has already held in a preceding iteration unless the coefficient vanishes a priori (if $o_2(t)-(\ell-1)k_1\leq o_2(0)$ holds). Thus, the coefficient $q_{i_2, k_2}$ is already known.\\
In the same vein assume that $o_2(t)\leq p_2=-i_1v-k_1\ell<o_2(t)+\ell-1$ holds in step \ref{algo:algo_wl_ord_2}. It follows
\begin{align*}
& o_2(t)+(\ell-1)k_1\leq -i_1v-k_1<o_2(t)+(\ell-1)(k_1+1)\\
\Rightarrow\quad
& o_1(t)+(\ell-1)(k_1+1-v)\leq -i_1v-k_1 < o_1(t)+(\ell-1)(k_1+2-v).
 \end{align*}
Then $k_1\leq v-1$ implies $k_1+2-v\leq 1$ and thus $-i_1v-k_1=p_1$ has to hold at the latest in the $t+1$-th iteration. Since loop \ref{algo:algo_wl_schleife_1} containing the calculations for $s_1$ is performed before loop \ref{algo:algo_wl_schleife_2} the coefficient $q_{i_1, k_1}$ is already known unless it vanishes.
\end{proof}

To compute the rational expression using $a_\ell$ we apply algorithm \ref{algo:algo_ausdruck_wl} to the functions $\tau_{\ell, n}^{(1)}, \tau_{\ell, n}^{(2)}$. Using lemma 2.23 from \cite{ELGS} and equation \eqref{eq:x_def} an easy calculation shows
\begin{equation}\label{eq:wl_x}
w_\ell(x(\zeta_\ell^t, q))=(\ell\tau)^2x(q^t, q^\ell)
\end{equation}
for $1\leq t\leq \ell-1$, an analogue statement holds for $y(\zeta_\ell^t, q)$. Hence, $\ord(x^*(\zeta_\ell^t, q))\geq 1$, $\ord(y^*(\zeta_\ell, q))\geq 1$, which implies
\begin{equation}\label{eq:ord_gs_wl_zerlegt}
\ord\left(\tau_{\ell, n}^{(1)}\right)\geq n-\ell e_{\Delta}, \quad \ord\left(\tau_{\ell, n}^{(2)}\right)\geq n-\ell e_{\Delta}-\ord(m_\ell).
 \end{equation}

\begin{koro}\label{koro:laufzeit_ausdruck_Al}
Algorithm \ref{algo:algo_ausdruck_wl} computes $Q$ using $\tilde{O}((v+e_\Delta)^2\ell^2)$ multiplications in $\Z$.
\end{koro}
\begin{proof}
In each iteration of loop \ref{algo:algo_wl_schleife_1} the order of $s_1$ strictly decreases by the construction of the algorithm, which implies the loop is called at most $\prec(\ell, n)$ times. Furthermore, each call of loops \ref{algo:algo_wl_schleife_1} and \ref{algo:algo_wl_schleife_2} in turn requires a constant number of multiplications of Laurent series provided some intermediate results are stored. We remark that, due to the values assumed by $i$ and $k$ in loop \ref{algo:algo_wl_schleife_2}, this has to be done in a clever way in order to obtain an efficient implementation. The Laurent series to be multiplied are computed up to precision $\prec(\ell, n)$. Using \eqref{eq:ord_gs_wl_zerlegt} and step \ref{algo:algo_wl_prec} of the algorithm to deduce
\begin{equation}\label{eq:prec}
\prec(\ell, n)=\ell e_\Delta-n+\ord(m_\ell)+(\ell+1)v=O(\ell(e_\Delta+v))
 \end{equation}
and taking into account that intermediate results are multiplied by the factor from corollary \ref{koro:gs_nenner} the assertion follows.
\end{proof}

Equation \eqref{eq:prec} shows the precision required for finding the rational expression and thus the run-time of all partial computations depend on the value of $v$. Apart from the functions  $m_\ell$ and in particular $a_\ell$ one might conceive using further alternatives. One approach to find such functions which is due to an idea of Atkin may be found in \cite[pp.~262--265]{Morain}. However, this procedure does not seem to have been much used in former computations, since it is relatively complicated and does not easily lend itself to the construction of a general algorithm. We remark that for any function $f_\ell$ one might use as an alternative the results from \cite{Abramovich} imply the lower bound
\begin{equation}\label{eq:untere_schranke_v}
\left|\ord(f_\ell)\right|=:v\geq \frac{7}{800}\ell
 \end{equation}
for the best possible values.

\section{Point-counting in the Elkies case}\label{sec:abschnitt_elkies}
\subsection{Gau{\ss} sums}\label{sec:elkies_gs}
In this section we give some details on how the representation
\begin{equation}\label{eq:gs_allgemein}
\frac{G_{\ell, n, \chi}(q)^n p_1(q)^r}{\Delta(q)^{e_{\Delta}}}=R_1(a_\ell(q), j(q))+\frac{R_2(a_\ell(q), j(q))}{g(q)}
 \end{equation}
precomputed by means of algorithms \ref{algo:algo_zerlegung_wl} and \ref{algo:algo_ausdruck_wl} may be used for counting points on an elliptic curve $E: Y^2=X^3+aX+b$ over a finite field $\F_p$ having $j$-invariant different from $0$ and $1728$.\\

As shown in detail in \cite{ELGS}, for an Elkies prime $\ell$ equation \eqref{eq:gs_allgemein} translates to the formula
\begin{equation}\label{eq:gs_konkret}
\frac{G_{\ell, n, \chi}(E)^n p_1(E)^r}{\Delta(E)^{e_{\Delta}}}=R_1(a_\ell(E), j(E))+\frac{R_2(a_\ell(E), j(E))}{g(E)}
\end{equation}
in terms of values associated to $E$. Here $G_{\ell, n, \chi}(E)$ is the elliptic Gau{\ss} sum defined by equation \eqref{eq:ell_gs}, $j(E)$ and $\Delta(E)$ are the well-known $j$-invariant and discriminant of $E$, $p_1(E)$ may be computed using the formulae from \cite[pp.~269--271]{Morain} and $a_\ell(E)$ is found as a root of the polynomial $A_\ell(X, j(E))$. Computing these values one directly obtains $G_{\ell, n, \chi}(E)^n$, which yields the index modulo $n$ in $\F_\ell^*$ of the eigenvalue $\lambda$ of the Frobenius homomorphism $\phi_p$ using equation \eqref{eq:lambda_aus_ell_gs} provided $p \equiv 1 \mod n$ holds. It is obvious that having determined this value modulo all maximal prime power divisors $n$ of $\ell-1$ we obtain $\lambda$ by using the Chinese Remainder Theorem, from which we glean $t \mod \ell$ from equation \eqref{eq:char_gl} as needed in Schoof's algorithm. Proceeding in this way to use equation \eqref{eq:lambda_aus_ell_gs} we avoid passing through large extensions of $\F_p$, which would be necessary when computing $G_{\ell, n, \chi}(E)$ directly from its definition \eqref{eq:ell_gs}. The advantage stems from \eqref{eq:ell_gs_pot_eigenschaft} stating that $G_{\ell, n, \chi}(E)^n$ lies in a much smaller extension than $G_{\ell, n, \chi}(E)$.\\

Computing the roots of the polynomial $A_\ell(X, j(E))$ in $\F_p$ yields two possible values for $a_\ell(E)$ corresponding to the two eigenvalues $\lambda$ and $\mu$ of $\phi_p$ from section \ref{sec:ell_kurven}. As already remarked below algorithm \ref{algo:algo_zerlegung_wl}, in equation \eqref{eq:gs_konkret} we choose $g=m_\ell^*-m_\ell$ and precompute $g^2$ as a rational expression in terms of $a_\ell$ and $j$ by means of algorithm \ref{algo:algo_ausdruck_wl}. This again translates to a formula on $E$ for $g(E)^2$,  which yields two possible values for $g(E)$ after a root  extraction.

Since $m_\ell^*=\frac{\ell^s}{m_\ell}-m_\ell$ for $s$ as in \eqref{eq:ml_def} holds, the correct one among the two candidates $\pm g(E)$ can be determined by solving the equations
\[ 
\frac{\ell^s}{x}-x=\pm g(E)
 \]
and checking for all solutions $x$ whether they are roots of $M_\ell(X, j(E))$. This yields the value of $m_\ell(E)$ corresponding to $g(E)$ at the same time. It would also be conceivable to determine $m_\ell(E)$ as a root of $M_\ell(X, j(E))$ and then to directly compute the value $g(E)$. However, the root-finding step turns out to have a significantly higher run-time than the approach just presented.\\
Once $a_\ell(E)$ or $m_\ell(E)$ are computed, the value $p_1(E)$ is determined using the formulae from \cite[pp.~102--106]{Mueller} or \cite[pp.~269--271]{Morain}, respectively. Since the first approach again requires finding the roots of some polynomial we prefer the second  variant on grounds of performance.\\
We remark that the value
$\frac{\partial A_\ell}{\partial Y}(a_\ell(E), j(E))$
used as the denominator of the different rational expressions computed by means of algorithm \ref{algo:algo_ausdruck_wl} may be zero in isolated cases, though this happens very rarely in practice. In this case we resort to the second possible value for $a_\ell(E)$. The same holds true for the values $p_1(E)$ and $g(E)=\frac{\ell^s}{m_\ell(E)}-m_\ell(E)$ (but \cite[p.~16]{ELGS} implies $m_\ell(E)\neq 0$).

\subsection{Jacobi sums}\label{sec:elkies_js}
\subsubsection{Theory}\label{sec:js_theorie}
As was remarked above, the approach using the (universal) elliptic Gau{\ss} sums only works if $p \equiv 1 \mod n$ holds. In order to be able to use formula \eqref{eq:lambda_aus_ell_gs} for arbitrary primes $p$, the Jacobi sums $\frac{G_{\ell, n, \chi}(E)^m}{G_{\ell, n, \chi^m}(E)}$ have to be determined as well. Directly using the function $\frac{G_{\ell, n, \chi}(q)^m}{G_{\ell, n, \chi^m}(q)}$ to construct a modular function of weight $0$ for $\Gamma_0(\ell)$ along the lines of \cite[Corollary 2.25]{ELGS} is not possible, since this function would not be holomorphic on $\mathbb{H}$. However, both \cite[Proposition 2.16]{ELGS} and proposition \ref{prop:darst_a_l} crucially rely on the holomorphicity of the functions $g$, $\tau_{\ell, n}$ for proving the existence of a rational expression of a special form well-suited for efficient computations. Hence, we slightly rearrange the key equation \eqref{eq:lambda_aus_ell_gs}. To this end, we remark that writing $m^{\prime}=n-m$ we obtain
\begin{equation}\label{eq:m_strich}
p=n(q+1)-m^{\prime},\quad m \equiv -m^{\prime} \mod n,
 \end{equation}
which yields the new equation
\begin{align}\label{eq:gleichung_konkret_3}
\chi^{-m}(\lambda)=\frac{(G_{\ell,n, \chi}(E)^n)^{q+1}}{G_{\ell,n, \chi}(E)^{m^{\prime}}
G_{\ell,n, \chi^{-m^{\prime}}}(E)}.
\end{align}
In the proof of corollary 2.24 of \cite{ELGS} it is shown that for $\gamma \in \Gamma_0(\ell)$
\[ 
G_{\ell, n, \chi}(q(\gamma\tau))^{k}=(c\tau+d)^{ek} \chi^{-k}(d)G_{\ell, n, \chi}(q)^k
 \]
holds, where $e=2$ for $n$ odd and $e=3$ for $n$ even holds. For this reason 
\[ 
G_{\ell, n, \chi}(q)^{k}G_{\ell,n, \chi^{-k}}(q)
 \]
is a modular function of weight $e(k+1)$ for $\Gamma_0(\ell)$. In particular, we obtain the following
\begin{lemma}\label{lem:js_def}
Let $\ell$ be a prime, $n \mid \ell-1$, $\chi: \F_\ell^* \rightarrow \mu_n$ be a character of order $n$ and let $k \in \N$. If $n$ is even, let $k$ be odd. Furthermore, let 
\[ 
r=\begin{cases}
\min\{r: \frac{k+1+r}{6} \in \N \},& n\equiv 1 \mod 2,\\
\min\{r: \frac{3(k+1)+2r}{12}\in \N \},& n\equiv 0 \mod 2
\end{cases}
\quad \text{and}\quad 
e_{\Delta}=\begin{cases}
\frac{k+1+r}{6}, & n\equiv 1 \mod 2,\\
\frac{3(k+1)+2r}{12}, & n \equiv 0 \mod 2.
\end{cases}
 \]
Then
\[ 
J_{\ell,n, \chi, k}(q)=J_{\ell, n, k}(q)=\frac{G_{\ell, n, \chi}(q)^{k}G_{\ell,n, \chi^{-k}}(q)p_1(q)^r}{\Delta(q)^{e_{\Delta}}}
 \]
is a modular function of weight $0$ for $\Gamma_0(\ell)$ which is holomorphic on $\mathbb{H}$ and whose coefficients lie in $\Q[\zeta_{n}]$. We call $J_{\ell, n, k}(q)$ a \textup{universal elliptic Jacobi sum}.
\end{lemma}
\begin{proof}
Using the above considerations the proof proceeds in exactly the same way as the one for corollary 2.24 of \cite{ELGS}. For even $n$ the condition on $k$ guarantees the existence of a suitable value for $r$.
\end{proof}

From our theory it follows that $J_{\ell,n, k}(q)$ admits a representation as a rational expression $R_k$ in terms of $j(q)$ as well as $m_\ell(q)$ and $a_\ell(q)$, respectively. This expression can be determined using the algorithms presented in section \ref{sec:implementierung} and \ref{sec:erg} for computing the universal elliptic Gau{\ss} sums. It is evident from equation \eqref{eq:wl_x} that $\ord(J_{\ell, n, k}^*)\geq (k+1)-\ell e_{\Delta}$ holds. Hence, we obtain
\begin{equation}\label{eq:ord_js}
\ord\left(J_{\ell, n, k}^{(1)}\right)\geq (k+1)-\ell e_{\Delta}, \quad 
\ord\left(J_{\ell, n, k}^{(2)}\right)\geq (k+1)-\ell e_{\Delta}-\ord(m_\ell)
\end{equation}
when applying algorithm \ref{algo:algo_ausdruck_wl}.\\

Having computed the rational expression $R_k$, one can determine the value
\[ 
J_{\ell, n, \chi, k}(E)=G_{\ell, n, \chi}(E)^{k}G_{\ell,n, \chi^{-k}}(E)
 \] 
in the same vein as in the previous section \ref{sec:elkies_gs}. In order to determine the index of $\lambda$ in $(\Z/\ell\Z)^*$ modulo $n$, we thus proceed as follows:

\begin{algorithm}[Determining the index of $\lambda$ modulo $n$]
\label{algo:algo_index_lambda}
\begin{algorithmic}[1]
    \REQUIRE{$\ell, n, E$}
    \ENSURE{Index of $\lambda$ in $(\Z/\ell\Z)^*$ modulo $n$}

	\STATE Determine $G_{\ell,n, \chi}(E)^n$ using equation \eqref{eq:gs_konkret}.\\
	\STATE Determine $J_{\ell, n, \chi, m^{\prime}}(E)$, where $m^{\prime}$ is as in equation \eqref{eq:m_strich}.\\
	\STATE Determine the index of $\lambda$ using equation \eqref{eq:gleichung_konkret_3}.\\
\end{algorithmic}
\end{algorithm}

We remark that in the representation $p=nq+m$ we obviously have $(m, n)=1$. In particular,  $m$ and $m^{\prime}$ are odd if $n$ is even, which implies the second step is only performed for values of $m^{\prime}$ covered by lemma \ref{lem:js_def}.\\

Applying this method we can use equation \eqref{eq:lambda_aus_ell_gs} for all maximal prime power divisors $n$ of $\ell-1$, which yields the value of $\lambda$ by means of the Chinese Remainder Theorem and thus $t$ modulo $\ell$. Having computed this value for sufficiently many primes $\ell$, we can determine the value of $t$ and finally that of $\#E(\F_p)$ using again the CRT as in Schoof's algorithm.

\subsubsection{Implementation}\label{sec:js_impl}
When implementing the computation of the Laurent series of the Jacobi sums and the determination of the rational expression we again avail ourselves of the ideas exposed in section \ref{sec:implementierung}. In particular, the expressions $G_{\ell, n, \chi^k}$ are multiplied by suitable cyclotomic Gau{\ss} sums before their product is computed. In this way all multiplications can again be performed in $\Q[\zeta_n]$ instead of in $\Q[\zeta_\ell, \zeta_n]$, which accounts for a significant run-time reduction.\\
For fixed $\ell$, $n$ all required Jacobi sums $J_{\ell, n, \chi, k}(q)$ are computed successively. As mentioned, these only have to be computed for $k$ coprime to $n$. Furthermore, equation \eqref{eq:lambda_aus_ell_gs} directly implies that for $m=1$, i.~e. $m^{\prime}=n-1$, no Jacobi sum is needed. Hence, our computation is as follows:

\begin{algorithm}[Computing the Jacobi sums corresponding to $\ell$ and $n$]
\label{algo:algo_js_lr}
\begin{algorithmic}[1] 
    \REQUIRE{$\ell, n, \prec(\ell, n)$}
    \ENSURE{Jacobi sums $J_{\ell, n, \chi, k}$ for $1\leq k \leq n-2$ and $(k, n)=1$}

	\STATE Determine $T=T_1:=G_{\ell, n, \chi}(q)G_{\chi^{-1}}(\zeta_\ell)$, $S=S_1:=G_{\chi^{-1}}(\zeta_\ell)$ up to precision $\prec(\ell, n)$.\\
	\FOR {$k=1$ \TO $n-2$}
		\STATE If $(k, n)>1$, go to step \ref{algo:js_algo_aktualisieren}.\\
		\STATE Determine $T_2:=G_{\ell, n, \chi^{-k}}(q)G_{\chi^k}(\zeta_\ell)$, $S_2:=G_{\chi^k}(\zeta_\ell)$. \label{algo:js_algo_wert_k}\\
		\STATE $T_3:=TT_2$, $S_3:=SS_2$. \label{algo:js_algo_mul}\\
		\STATE $T_4:=T_3S_3^{-1}$.\\
		\STATE Compute $J_{\ell,n, \chi, k}(q)$ by multiplication of $T_4$ by suitable powers of $p_1(q)$ and $\Delta(q)$.\\
		\STATE $T:=TT_1$, $S:=SS_1$.\label{algo:js_algo_aktualisieren}\\
	\ENDFOR
\end{algorithmic}
\end{algorithm}

Step \ref{algo:js_algo_aktualisieren} obviously guarantees that $T=T_1^k$ as well as $S=S_1^k$  hold in each iteration in step \ref{algo:js_algo_mul}, which proves the correctness. Let $c$ be a generator of $(\Z/n\Z)^*$. Since $(\ell, n)=1$ holds, we obtain
$\gal(\Q[\zeta_\ell, \zeta_n]/\Q[\zeta_\ell])=\langle \sigma: \zeta_n \mapsto \zeta_n^c\rangle$. As $\sigma$ is a homomorphism, it follows
\[ 
\sigma(G_{\ell, n, \chi}(q))=G_{\ell, n, \chi^c}(q)\quad \text{and}\quad \sigma(G_{\chi^{-1}}(\zeta_\ell))=G_{\chi^{-c}}(\zeta_\ell).
 \]
Thus, the expression in step \ref{algo:js_algo_wert_k} can be recovered from the precomputed values $T_1$, $S_1$ with negligible costs by applying the homomorphism $\zeta_n \mapsto \zeta_n^k$.\\

The most costly step inside the loop is the computation of $T_3$ and the updating of $T_1$. Each of these requires $O(\M(n\prec(\ell, n)))$ operations. Using equations \eqref{eq:prec_algo_ausdruck} and \eqref{eq:ord_js} we see that as in \eqref{eq:prec} it suffices to take
\begin{equation}\label{eq:prec_js}
\prec(\ell, n)=\ell(e_\Delta+v+1)
 \end{equation}
when afterwards applying algorithm \ref{algo:algo_ausdruck_wl} to compute the rational expression in terms of $a_\ell$ and $j$.

\subsection{Run-time and memory requirements}\label{sec:laufzeit_elkies}
We compute $G_{\ell, n, \chi}(E)^n$ using equation \eqref{eq:gs_konkret}. Once the values of $j, a_\ell, \Delta$, $p_1$ and $g$ on $E$ have been determined as detailed in section \ref{sec:elkies_gs}, the evaluation of the right hand side of this equation requires a further  $\prec(\ell, n)$ operations to determine $R_i(a_\ell, j)$, $i=1,2$. As shown in equation \eqref{eq:prec}, the value $\prec(\ell, n)$, which provides a bound on the degree of enumerator and denominator of $R$ in terms of $j$ and $a_\ell$, can be chosen to be
$\prec(\ell, n)=(v+e_\Delta+1)\ell.$ Since the expression $R(a_\ell, j)$ lies in $\F_p[\zeta_n]$, the cost for computing $G_{\ell, n, \chi}(E)^n$ using \eqref{eq:gs_konkret} amounts to $O((v+e_\Delta+1)\ell\M(n))$ multiplications in $\F_p$. The computation of $J_{\ell, n, \chi, m^{'}}(E)$ requires comparable costs according to section \ref{sec:js_impl}. It is easy to see that these costs dominate the work for precomputing the values $j, m_\ell, \Delta, p_1, g$. Subsequently, according to equation \eqref{eq:gleichung_konkret_3} the essential work consists in determining the power $(G_{\ell, n, \chi}(E)^n)^{q+1}$, which requires $O(\M(n)\log q)=O(\M(n)\log p)$ operations, since $n \ll p$ holds. Hence, the total run-time amounts to
\begin{equation}\label{eq:laufzeit_neu}
 O(\M(n)((v+e_{\Delta}+1)\ell+\log p)).
\end{equation}
We compare this to the algorithm presented in \cite{MiMoSc}, whose run-time is given by
\[ 
O\left(\Cpoly(\ell)\log \frac{\ell}{n} + \M(n)\log p + \Cpoly_{\sqrt{n}}(n)\right)
 \]
using the notation from that article, which was one of the starting points of our research. Here
\begin{equation}\label{eq:laufzeit_mms}
O\left(\Cpoly(\ell)\log \frac{\ell}{n} +\Cpoly_{\sqrt{n}}(n)\right) =\tilde{O}\left(\ell^{\frac{\omega+1}{2}}+n^{\frac{3\omega+1}{4}}\right)
 \end{equation}
holds, where $2\leq \omega<3$ is an exponent for matrix multiplication in $\F_p$. The theoretical record is $\omega \simeq 2.38$ from \cite{CoWi}. Comparing \eqref{eq:laufzeit_neu} and \eqref{eq:laufzeit_mms} we realise that our approach might be competitive provided $n$ and $v$ are comparatively small. Due to $e_\Delta \approx \frac{n}{6}$ one should at least require $n\leq\sqrt{\ell}$.\\
Recent results by Tenenbaum in \cite{Tenenbaum}, which build on the well-known asymptotic formulae for smooth numbers in \cite{deBruijn}, show
\[
\Upsilon(x, y)=\#\{z \leq x: p^k\ ||\ z \Rightarrow p^k\leq y\} \sim x\rho(u)\quad \text{for } x \rightarrow \infty,\quad \text{where }x=y^u,
\]
for the count of $y$-ultrafriable numbers $\leq x$. Here  $\rho(u)$ denotes the Dickmann function. For $u=2$ one gleans $\rho(u)\approx 0.307$. Thus, we can expect that for about 30\% of the primes $\ell$ we consider the values of $n$ are small enough that an improvement of the run-time might be possible.\\
However, since equation \eqref{eq:untere_schranke_v} implies $v=O(\ell)$ holds asymptotically, we find that our algorithm exhibits an asymptotic run-time of $O(\ell^2\M(n))$ and is not competitive with the one from \cite{MiMoSc}.\\

One might observe that the run-time for the computation of the $q$-th power of $G_{\ell, n}(E)^n$ can be reduced if the $n$-th cyclotomic polynomial is reducible over $\F_p$. In this case the extension $\F_p[\zeta_n]/\F_p$ only has degree $\min\{k: p^k \equiv 1 \mod n\}=\ord_n(p)$. This observation can account for an improvement in run-time merely in a few cases, though.\\

However, the memory requirements are much more forbidding than run-time when considering practical applications for counting points on elliptic curves. As follows from the observations in section \ref{sec:erg}, the polynomial $Q$ corresponding to $\tau_{\ell,n}(q)$ contains $\prec(\ell, n)$ coefficients from $\Q[\zeta_n]$. Experimental results suggest the height (the logarithm of the maximal absolute value) of these coefficients is essentially proportional to $v$ (similar results are well-known from \cite{Cohen} for the modular function $j(q^\ell)$) and thus asymptotically proportional to $\ell$ according to \eqref{eq:untere_schranke_v}. Hence, they imply
\[ 
\tilde{O}(\prec(\ell, n)n\ell)=\tilde{O}((v+e_\Delta)n\ell^2)
 \]
bytes of memory are required in order to store all the coefficients of $Q$. Using again $v=O(\ell)$ as well as $n=O(\ell)$ in the worst case we obtain a memory requirement of $\tilde{O}(\ell^4)$ bytes. This means in the worst case for $\ell \approx 100$ the expected memory requirement for representing the polynomial $Q$  amounts to about $100$ MB, which is confirmed by our computations. 
The expected memory requirement rises to $1.6$ GB for $\ell \approx 200$ and to about $1$ TB for $\ell \approx 1000$. Taking into account that according to section \ref{sec:js_theorie} the polynomial $Q$ for fixed $\ell$ and $n$ has to be precomputed for $\varphi(n)-1$ different Jacobi sums in order to count points on elliptic curves $E/\F_p$ for arbitrary primes $p$, it is evident that the proposed method rapidly leaves the realms of possibility. We observe that for the records set in \cite{Rekord} prime numbers up to $\ell \approx 4000$ and for those in \cite{Sutherland} even primes up to $\ell \approx 11000$ were used.

\footnotesize{
\bibliography{lit}}

\end{document}